\documentclass[12pt]{article}
\usepackage[utf8]{inputenc}
\usepackage[T1]{fontenc}
\usepackage{color}
\usepackage{tocstyle}
\usepackage{graphicx,makeidx}
\usepackage{epsf}
\usepackage{amsmath,amsfonts,amssymb,latexsym}
\usepackage{enumitem}
\usepackage[width=160mm,height=230mm,left=35mm,foot=10mm]{geometry}
\usepackage{setspace}
\usepackage{mathtools}

\def\fakeend{\end{document}}

\usepackage{chngcntr}
\counterwithin{figure}{section}
\makeindex

\newcommand{\ignore}[1]{}
\newcommand{\startClaims}{\setcounter{claim}{0}}

\newtheorem{theorem}{Theorem}[section]

\newtheorem{lemma}[theorem]{Lemma}

\newtheorem{claim}{Claim}

\makeatletter
\let\c@figure\c@theorem
\makeatother

\makeatletter
\@addtoreset{figure}{section}
\makeatother

\title{Drawings of $K_n$ with the same rotation scheme\\ are the same up to Reidemeister moves\\ (Gioan's Theorem)\\ {\bf\small\textsc{In memory of our friend Dan Archdeacon.}}}

\author{Alan Arroyo\footnote{Supported by CONACYT. (GS by grant no.\ 222667)}\ ${}^+$,
\ Dan McQuillan${}^\ddagger$,\\  R.\ Bruce Richter\footnote{Supported by NSERC grant no.\ 41705-2014 057082.
\newline{\textcolor{white}{...}}
${}^+$University of Waterloo, 
${}^\ddagger$Norwich University, and ${}^\times$UASLP}\ ${}^+$, and Gelasio Salazar${}^*$${}^\times$}

\date{\LaTeX-ed: \today}

\newenvironment{proof}%
{\noindent{\bf Proof.}\ }%
{\hfill\eopf\par\bigskip}%

{\noindent{\bf Proof of #1.}\ }%
{\hfill\eopf\par\bigskip}%

{\noindent{\bf Sketch of a proof.}\ }%
{\hfill\eopf\par\bigskip}%
{\noindent{\bf Ideas for a proof.}\ }%
{\hfill\eopf\par\bigskip}%
{\noindent{\bf Proof in progress.}\ }%
{\hfill\eopf\par\bigskip}%

\def\eop{\hfill{{\rule[0ex]{7pt}{7pt}}}}

\newenvironment{cproofof}[1]
{\bigskip\noindent{\bf Proof of #1.}\startClaims\ }
{\hfill{\eop}\par\bigskip}

\newenvironment{cproof}
{\noindent{\bf Proof.}\startClaims\ }
{\hfill{\eop}\par\bigskip}

\def\i4c{{inter\-nally-4-con\-nec\-ted}}
\def\p4c{\wording{peri\-phe\-rally-4-con\-nec\-ted}}
\def\ei4c{\operatorname{\wording{p4c}}}

\def\2cc{{2-cros\-sing-cri\-tical}}

\def\m2{{{\cal M}_2^3}}

\newcommand{\eopf}{\raisebox{0.8ex}{\framebox{}}}

\newcommand{\majorrem}[1]{}
\newcommand{\minorrem}[1]{}
\newcommand{\wording}[1]{#1}
\newcommand{\wordingrem}[1]{}
\newcommand{\dragominorrem}[1]{}

\def\rtwo{\mathbb R^2}

\def\r5{r^*_{+5}}

\def\rightspine #1{{}_{#1}\kern-3pt\sqsubset}

\newcommand{\changes}[1]{#1}
\newcommand{\changer}[1]{#1}

\begin{document}

\maketitle

\begin{abstract} A {\em good drawing\/} of $K_n$ is a drawing of  the complete graph with $n$ vertices in the sphere such that:  no two edges with a common end  cross;  no two edges cross more than once; and no three edges all cross at the same point.  Gioan's Theorem asserts that any two good drawings of $K_n$ that have the same rotations of incident edges at every vertex are equivalent up to Reidemeister moves.   At the time of preparation, 10 years had passed between the statement in the WG 2005 conference proceedings and our interest in the proposition.  Shortly after we completed our preprint, Gioan  independently completed a preprint.  
\end{abstract}
\baselineskip =18pt

\section{Introduction}\label{sec:intro}

The main result of this work is the proof of the following result, presented by Gioan at the International Workshop on Graph-Theoretic Concepts in Computer Science 2005 (WG 2005) \cite{gioan}.

\begin{theorem}[Gioan's Theorem]\label{th:gioan}
Let $D_1$ and $D_2$ be good drawings (defined below) of $K_n$ in the sphere that have the same rotation schemes.  Then there is a sequence of Reidemeister moves (example below, defined in Section \ref{sec:gioan}) that transforms $D_1$ into $D_2$.
\end{theorem}

We are only using ``Reidemeister III'' moves to shift a bit of the interior of an edge across another crossing (without crossing anything else).   Figure \ref{fg:vxInReidTriangle}  shows a typical example of ``before'' and ``after'' the move.

\begin{figure}[!ht]
\begin{center}
\includegraphics[scale=.4]{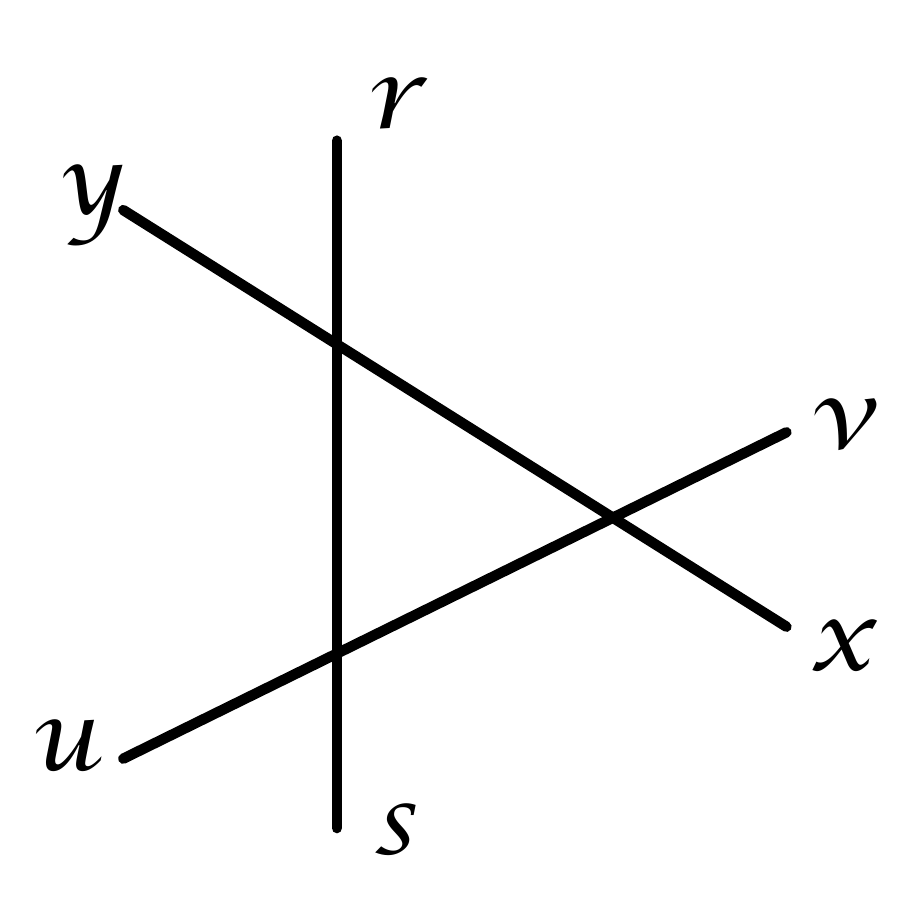}\hskip .5truein
\includegraphics[scale=.4]{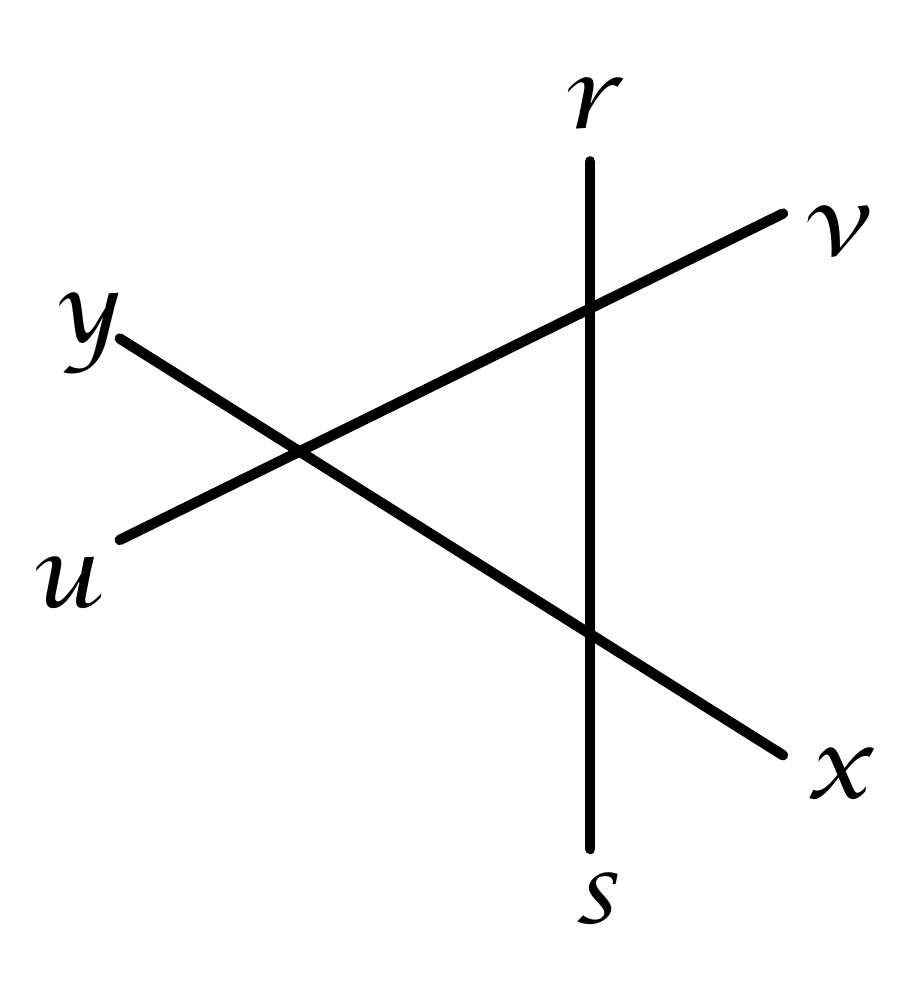}
\caption{A Reidemeister III move that transforms one drawing into another.}\label{fg:vxInReidTriangle}
\end{center}
\end{figure}


The Harary-Hill Conjecture asserts that the crossing number of the complete graph $K_n$ is equal to 
\[
H(n) := \frac 14\left\lfloor\frac{\mathstrut n}{\mathstrut 2}\right\rfloor
\left\lfloor\frac{\mathstrut n-1}{\mathstrut 2}\right\rfloor
\left\lfloor\frac{\mathstrut n-2}{\mathstrut 2}\right\rfloor
\left\lfloor\frac{\mathstrut n-3}{\mathstrut 2}\right\rfloor \,.
\]

Throughout this work, all drawings of graphs are {\em good drawings\/}: \begin{itemize}[leftmargin = .6truein]
\item no two edges incident with a common vertex cross; \item no two edges cross each other more than once; and \item no three edges cross at a common point.
\end{itemize}

Some of our interest in this problem derives from Dan Archdeacon's combinatorial generalization of this problem. Since his website may soon be lost and there is no other version that we know of, we reproduce it here.

\begin{changemargin}{1cm}{1cm} 
{\em Suppose the vertex set of $K_n$ is $I_n = \{1,...,n\}$. A local neighborhood of a vertex $k$ in a planar drawing determines a cyclic permutation of the edges incident with $k$ by considering the clockwise ordering in which they occur. Equivalently (looking at the edges' opposite endpoints), it determines a local rotation $\rho(k)$: a cyclic permutation of $I_n - k$. A (global) rotation is a collection of local rotations $\rho(k)$, one for each vertex $k$ in $I_n$.

It is well known that the rotations of $K_n$ are in a bijective correspondence with the embeddings of $K_n$ on oriented surfaces. The rotation arising from a planar drawing also determines which edges cross. Namely, edges $ab,cd$ cross in the drawing if and only if the induced local rotations on the vertices $\{a,b,c,d\}$ give a nonplanar embedding of that induced $K_4$.  {\em  [This is not quite true:  the rotation determines the crossing among the six edges in the $K_4$ induced by $a,b,c,d$, but it is not necessarily true that it is $ab$ with $cd$. AMRS]}

The stated conjecture on the crossing number of $K_n$ asserts that the minimum number (over all planar drawings) of induced nonplanar $K_4$\!'s satisfies the given lower bound. We generalize this to all rotations.

\bigskip
{\bf Conjecture:} In any rotation of $K_n$, the number of induced nonplanar $K_4$\!'s is at least $(1/4)\hskip 1pt [n/2]\hskip 1pt [(n-1)/2]\hskip 1pt [(n-2)/2]\hskip 1pt [(n-3)/2]$ where $[m]$ is the integer part of $m$.

\bigskip 
Not every rotation corresponds to a drawing (see the related problem ``Drawing rotations in the plane''), so this conjecture is strictly stronger than the one on the crossing number of $K_n$. However, this conjecture has the advantage of reducing a geometric problem to a purely combinatorial one.

The problem arose from my attempts to prove the lower bound on the crossing number. It is supported by computer calculations. Namely, I wrote a program which started with a rotation of $K_n$ and using a local optimization technique (hill-climbing), randomly swapped edges in a local rotation whenever that swap did not increase the number of induced nonplanar $K_4$\!'s. The resulting locally minimal rotations tended to resemble the patterns apparent in an optimal drawing of $K_n$. For small $n$ this minimum was the conjectured upper bound. For larger $n$ it was usually slightly larger.}

\end{changemargin}

It is well-known that the rectilinear crossing number (all edges are required to be straight-line segments) of $K_n$ is, for $n\ge 10$, strictly larger than $H(n)$ \cite{rectilinearVsUsual}.  In fact, this applies to the more general {\em pseudolinear\/} crossing number \cite{pseudolinearVsUsual}.  

An {\em arrangement of pseudolines\/} $\Sigma$ is a finite set of simple open arcs in the plane $\rtwo$ such that:  for each $\sigma\in\Sigma$, $\rtwo\setminus \sigma$ is not connected; and for distinct $\sigma$ and $\sigma'$ in $\Sigma$, $\sigma\cap \sigma'$ consists of a single point, which is a crossing.   

A drawing of $K_n$ is {\em pseudolinear\/} if there is an arrangement $\Sigma$ of $\binom n2$ pseudolines such that the edges of $K_n$ are all  contained in different pseudolines of $\Sigma$.    It is clear that a rectilinear drawing (chosen so no two lines are parallel) is pseudolinear.

The arguments (originally due to Lov\'asz et al \cite{lovasz} and, independently, \'Abrego and Fern\'andez-Merchant \cite{af-m}) that show every rectilinear drawing of $K_n$ has at least $H(n)$ crossings apply equally well to pseudolinear drawings.

The proof that every optimal pseudolinear drawing of $K_n$ has its outer face bounded by a triangle \cite{optTriang} uses the ``allowable sequence'' characterization of pseudoline arrangements of Goodman and Pollack \cite{goodPoll}.  Our principal result in \cite{amrs} is that there is another, topological, characterization of pseudolinear drawings of $K_n$.

A drawing $D$ of $K_n$ is {\em face-convex\/} if there is an open face $F$ of $D$ such that, for every 3-cycle $T$ of $K_n$, if $\Delta$ is the closed face of $D[T]$ disjoint from $F$, then, for any two vertices $u,v$ such that $D[u],D[v]$ are both in $\Delta$, the arc $D[uv]$ is also contained in $\Delta$. 

The main result in \cite{amrs} is that every face-convex drawing of $K_n$ is pseudolinear \changer{and conversely}.   An independent proof has been found by Aichholzer et al \cite{ahpsv}; their proof uses Knuth's  CC systems \cite{knuth}, which are an axiomatization of sets of pseudolines.    Moreover, their statement is in terms of a forbidden configuration.  Properly speaking, their result is of the form, ``there exists a face relative to which the forbidden configuration does not occur''.  Their face and our face are the same.  However, our proof is completely different, yielding directly a polynomial time algorithm for finding the pseudolines.

Aichholzer et al show that there is a pseudolinear drawing of $K_n$ having the same crossing pairs of edges as the given drawing of $K_n$.   Gioan's Theorem \cite{gioan} (Theorem \ref{th:gioan} above) is then invoked to show that the original drawing is also pseudolinear.     

The proof in \cite{amrs} is completely self-contained; in particular, it does not invoke Gioan's Theorem.  An earlier version anticipated an application of Gioan's Theorem similar to that in \cite{ahpsv}; hence our interest in having a proof.

A principal ingredient in our argument is a consideration of the facial structure of an arrangement of arcs in the plane.   An {\em arrangement of arcs\/} is a finite set $\Sigma$ of open arcs in the plane such that, for every $\sigma\in\Sigma$, $\rtwo\setminus \sigma$ is not connected and any two elements of $\Sigma$ have at most one point in common, which must be a crossing. 

Let $\Sigma$ be an arrangement of arcs.  Since $\Sigma$ is finite, there are only finitely many faces of $\Sigma$:  these are the components of $\rtwo\setminus (\bigcup_{\sigma\in \Sigma}\sigma)$.   As it comes up often, we let $\mathcal P(\Sigma)$ be the pointset $\bigcup_{\sigma\in\Sigma}\sigma$. 

The {\em dual\/} $\Sigma^*$ of $\Sigma$ is the finite graph whose vertices are the faces of $\Sigma$ and there is one edge for each segment $\alpha$ of each $\sigma\in \Sigma$ such that $\alpha$ is one of the components of $\sigma\setminus \mathcal P(\Sigma\setminus\{\sigma\})$.  The dual edge corresponding to $\alpha$ joins the faces of $\Sigma$ on either side of $\alpha$.

Although we do not need it here, the following lemma motivates one that we need in our proof of Gioan's Theorem.  Its simple proof from \cite{amrs} is included here for completeness.

\begin{lemma}[Existence of dual paths]\label{lm:dualPaths}  Let $\Sigma$ be an arrangement of arcs in the plane and let $a,b$ be points of the plane not in any line in $\Sigma$.  Then there is an $ab$-path in $\Sigma^*$ crossing each arc in $\Sigma$ at most once.
\end{lemma}

\begin{cproof}
We proceed by induction on the number of curves in $\Sigma$ that separate $a$ from $b$, the result being trivial if there are none.  Otherwise, for $x\in \{a,b\}$,  let $F_x$ be the face of $\Sigma$ containing $x$ and let $\sigma\in\Sigma$ be incident with $F_a$ and separating $a$ from $b$.  Then $\Sigma^*$ has an edge $F_aF$ that crosses $\sigma$.  

Let $R$ be the region of $\rtwo\setminus \sigma$ that contains $F_b$ and let $\Sigma'$ be the set $\{\sigma'\cap R\mid \sigma'\in \Sigma,\ \sigma'\cap R\ne\varnothing\}$.  The induction implies there is an $FF_b$-path in $\Sigma'{}^*$.  Together with $F_aF$, we have an $F_aF_b$-path in $\Sigma^*$, as required.
\end{cproof}

\section{Proof of Gioan's Theorem}\label{sec:gioan}

In this section, we give a simple, self-contained proof Gioan's Theorem \cite{gioan}.  When we completed the proof in August 2015, we corresponded with Gioan, who was independently preparing his own version.  Each version has had some impact on the other.  We do not include any of the first order logical considerations that occur in Gioan's version.

For convenience, we restate our main result here.  The definition of a Reidemeister move is given just after this statement.

\bigskip\noindent{\bf Theorem \ref{th:gioan}} {\em 
Let $D_1$ and $D_2$ be drawings of $K_n$ in the sphere that have the same rotation schemes.  Then there is a sequence of Reidemeister moves that transforms $D_1$ into $D_2$.}

\bigskip

In order to define Reidemeister move and prove our first intermediate lemmas, we require a small new consideration.  Let $\Sigma$ be an arrangement of arcs in the plane.   A {\em vertex of $\Sigma$\/} is a point that is the intersection of two or more arcs in $\Sigma$.   

At a vertex $v$, the rotation of the arcs containing $v$ is of the form $\sigma_1,\sigma_2,\dots,\sigma_k,\sigma_1,\sigma_2,$ $\dots,\sigma_k$; each arc occurs twice here, once for each of the ``rays'' it contains that start at $v$.   Let $(F_0,F_1,\dots,F_{k-1},F_k,F_{k+1},\dots,F_{2k-1})$ the cyclic sequence of faces around $v$.  

Suppose $P$ is a dual path containing the subpath $(F_0,F_1,\dots,F_k)$ such that $P$ crosses each arc in $\Sigma$ at most once.   The path obtained from $P$ by {\em sliding over the vertex $v$\/} is the path $P$, except $(F_0,F_1,\dots,F_k)$ is replaced by the dual path (of the same length) $(F_0,F_{2k-1},F_{2k-2},\dots,F_{k+1},F_k)$.  (None of $F_{2k-1},F_{2k-2},\dots,F_{k+1}$ can occur in $P$, as $P$ crosses each arc of $\Sigma$ at most once.  Thus, the result of the sliding is indeed a new dual path.)

We remark that we may interpret the change as either rerouting $P$ across $v$ or moving $v$ across $P$ and adjusting the edges incident with $v$.

A {\em Reidemeister move\/} is a sliding over a vertex $v$ that is in precisely two arcs in $\Sigma$.  The following may be viewed as a supplement to Lemma \ref{lm:dualPaths}.

\begin{lemma}\label{lm:reidemeister}
Let $\Sigma$ be an arrangement of arcs in the plane and let $a$ and $b$ be any two points in the plane not in $\mathcal P(\Sigma)$.  Let $F_a$ and $F_b$ be the faces of $\Sigma$ containing $a$ and $b$, respectively.  Then, any distinct $F_aF_b$-paths $P$ and $Q$ in $\Sigma^*$, each crossing every arc in $\Sigma$ at most once, are equivalent up to sliding over vertices.  Moreover, there is a sequence of slidings such that every sliding involves moving a vertex across $P$ from inside to outside, always relative to the closed disc bounded by $P\cup Q$.
\end{lemma}

\begin{cproof}
Let $P_1$ and $Q_1$ be subpaths of $P$ and $Q$ having common end points but otherwise disjoint.  Then (any natural image in the plane of) $P_1\cup Q_1$ bounds a disc $\Delta$ and each arc in $\Sigma$ that crosses one of $P_1$ and $Q_1$ crosses the other.  We will show that there is a vertex in $\Delta$ over which we can slide $P_1$.

Since $P_1$ and $Q_1$ are distinct dual paths, there is a vertex of $\Sigma$ in $\Delta$.  Let $\sigma\in \Sigma$ have an arc across $\Delta$ and contain a vertex of $\Sigma$; let $v$ be the first vertex of $\Sigma$  encountered as we traverse $\sigma$ across $\Delta$ from its $P_1$-end.      Among all the $\sigma\in \Sigma$ that contain $v$, either all have $v$ as their first encountered vertex or there are two, $\sigma$ and $\bar \sigma$, consecutive in the rotation at $v$,  such that $v$ is the first encountered vertex for $\sigma$, but not for $\bar\sigma$.  In the former case, we can slide $v$ across $P_1$.

Suppose $\sigma'\in\Sigma$ has a crossing with $\bar\sigma$ between $v$ and the intersection of $\bar\sigma$ with $P_1$.  Let $\Delta'$ be the disc bounded by $P_1$, $\sigma$, and $\bar\sigma$.  Then $\sigma'\cap \Delta'$ intersects the boundary of $\Delta'$ at least twice, but not on \changer{$\sigma\cap \Delta'$}.  Thus, $\sigma'$ crosses $P_1$ between $\sigma\cap P_1$ and $\bar\sigma\cap P_1$.  

Let $\bar v$ be the first vertex of $\Sigma$ encountered as we traverse $\bar\sigma$ from $\bar\sigma\cap P_1$.  Then every other arc in $\Sigma$ that contains $\bar v$ intersects $P_1$ between $\sigma\cap P_1$ and $\bar\sigma\cap P_1$.

Letting $b(v)$ denote the number of arcs in $\Sigma$ that cross $P_1$ between $\sigma\cap P_1$ and $\bar\sigma\cap P_1$, we see that $b(\bar v)<b(v)$.  Therefore, there is always a vertex $w$ of $\Sigma$ such that $b(w)=0$ and we can slide $w$ across $P_1$.

After sliding $w$ across $P_1$,  the disc bounded by $P_1$ and $Q_1$ has fewer vertices of $\Sigma$.  An easy induction completes the proof.
\end{cproof}

Gioan's Theorem considers two drawings $D_1 $ and $D_2$ of $K_n$ in the sphere that have the same rotation scheme.  Let $t,u,v,w$ be four distinct vertices of $K_n$.  Let $T$ be the 3-cycle induced by $t,u,v$.  Then $D_1[T]$ is a simple closed curve in the sphere.  The rotations at $t$, $u$, and $v$ determine where bits of the edges $D_1[tw]$, $D_1[uw]$, and $D_1[vw]$ go from their ends $t$, $u$, and $v$, respectively, relative to $D_1[T]$.  The side of $D_1[T]$ that has the majority (two or three) of these bits of edges is where $D_1[w]$ is.  If $tw$ is the minority edge, then $D_1[tw]$ crosses $D_1[uv]$; conversely, a crossing $K_4$ produces, for each of its 3-cycles, a minority edge.   This simple observation immediately yields the following fundamental fact.

\begin{enumerate}[label=(F\arabic*)]
\item \label{it:rotationK4}
Let $D_1$ and $D_2$ be two drawings of $K_n$ with the same rotation scheme.  If $J$ is any $K_4$ in $K_n$, then there is an orientation-preserving homeomorphism of the sphere to itself mapping $D_1[J]$ onto $D_2[J]$ that preserves the vertex-labels of $J$.
\end{enumerate}
 
There are some elementary corollaries of \ref{it:rotationK4}:
\begin{enumerate}[resume,label=(F\arabic*)]
\item\label{it:rotationCrossing}  the pairs of crossing edges are determined by the rotation scheme; 
\item\label{it:rotationDirCrossing} if the edges of $K_n$ are oriented, then the directed crossings are determined by the rotation scheme; and
\item\label{it:triangleContainment}  if $u,v,w,x$ are distinct vertices of $K_n$, then the side  of the 3-cycle (relative to any of its oriented sides) induced by $u,v,w$ that contains $x$ is determined by the rotation scheme.
\end{enumerate}
By \ref{it:rotationDirCrossing}, we mean that, if $e$ and $f$ cross, then, as we follow the orientation of $e$, the crossing of $e$ by the traversal $f$ is either left-to-right in all drawings or right-to-left in all drawings, depending only on the rotation scheme. 

These facts can hardly be new.  In fact, variations of some of them appear in Kyn\v cl \cite{kyncl}.

\begin{lemma}\label{lm:crossingOrderDetermines}  Let $D_1$ and $D_2$ be two drawings of $K_n$ in the sphere with the same rotation scheme.  Let $G$ be a subgraph of $K_n$ and suppose that, for each edge $e$ of $G$, as we traverse $e$ from one end to the other, the edges of $G$ that cross $e$ occur in the same order in both $D_1$ and $D_2$.  Then there is an orientation-preserving homeomorphism of the sphere mapping $D_1[G]$ onto $D_2[G]$ that preserves all vertex- and edge-labels.
\end{lemma}

\begin{cproof}
This is a consequence of the well known fact that a rotation scheme of a \changer{connected} graph determines a unique (up to surface orientation-preserving homeomorphisms) cellular embedding of a graph in an orientable surface; see \cite[Thm.~3.2.4]{mt}. We construct a planar map from each of $D_1[G]$ and $D_2[G]$ by inserting a vertex of degree 4 at each crossing point.  By \ref{it:rotationDirCrossing} and the hypothesis, respectively, the oriented crossings and the orders of the crossings of each edge are the same in both $D_1$ and $D_2$.  Thus, the rotations at these degree 4 vertices are also the same.  Therefore, the planar maps $D_1[G]$ and $D_2[G]$ are the same, as claimed.
\end{cproof}

Lemma \ref{lm:crossingOrderDetermines} asserts that the orders of crossings determine the drawing.  Thus, we need to consider the situation that some edge has two edges crossing it in different orders in the two drawings.  

Let $e$, $f$, and $g$ be three distinct edges in a drawing $D$ of $K_n$, no two having a common end.  Suppose each two of $e$, $f$, and $g$ have a crossing, labelled $\times_{e,f}$, $\times_{e,g}$, and $\times_{f,g}$.  The union of the segments of each of $e$, $f$, and $g$ between their two crossings is a simple closed curve.  If one of the two sides of this simple closed curve does not have an end of any of $e$, $f$, and $g$, then this closed disc is the  {\em pre-Reidemeister triangle constituted by $e$, $f$, and $g$.}  

Let $D_1$ and $D_2$ be drawings of $K_n$ in the sphere with the same rotation scheme.  A {\em Reidemeister triangle\/} for $D_1$ and $D_2$ is a pre-Reidemeister triangle $T$ for both $D_1$ and $D_2$ constituted by the edges $e$, $f$, and $g$ but with the clockwise traversal of the three segments between pairs of crossings giving the opposite cyclic ordering of the three crossings. 

Let $J$ be a $K_4$ in $D_1$ with a crossing.  Then \ref{it:rotationCrossing} shows that $D_2[J]$ also has a crossing, with the same pair of edges crossing.  For $\alpha\in\{1,2\}$, let $\times^\alpha$ denote the crossing in $D_\alpha[J]$.  Then $D_\alpha[J]$ has five faces:  one {\em 4-face\/} bounded by a 4-cycle in $J$; and four {\em 3-faces\/}, each incident with $\times^\alpha$.  

\bigskip\noindent{\bf Notation}  If $x$ and $r$ are the two vertices of $J$ incident with a 3-face that with crossing edges $e$ and $f$, then \changer{we use $T^\alpha_{x,r}$ to denote this 3-face and $xr\times^\alpha_{e,f}$ to denote its boundary}.  

\bigskip

Our next lemma corresponds to Lemma 3.2 of \cite{gioan}.  This result is a central, non-trivial point in the argument.

\begin{lemma}\label{lm:differentOrderReidTriang}
Let $D_1$ and $D_2$ be two drawings of $K_n$ in the sphere with the same rotation scheme.   Then, for any Reidemeister triangle $R$ for $D_1$ and $D_2$, $D_1[R]$ contains a vertex of $D_1[K_n]$.
\end{lemma}

\begin{cproof}  
Let $R$ be a Reidemeister triangle in $D_1[K_n]$ for $D_1$ and $D_2$. We use the same labelling $e=xy$, $f=uv$, and $g=rs$ as above for the edges determining $R$; all of $r$, $s$, $u$, $v$, $x$, and $y$ are in the same face $F$ of $D_1[R]$.  By way of contradiction, suppose there is a vertex $a$ of $K_n$ in the other face $F_a$ of $D_1[R]$.     See the left-hand figure in Figure \ref{fg:vxInReidTriangle1}.

\begin{figure}[!ht]
\begin{center}
\includegraphics[scale=.4]{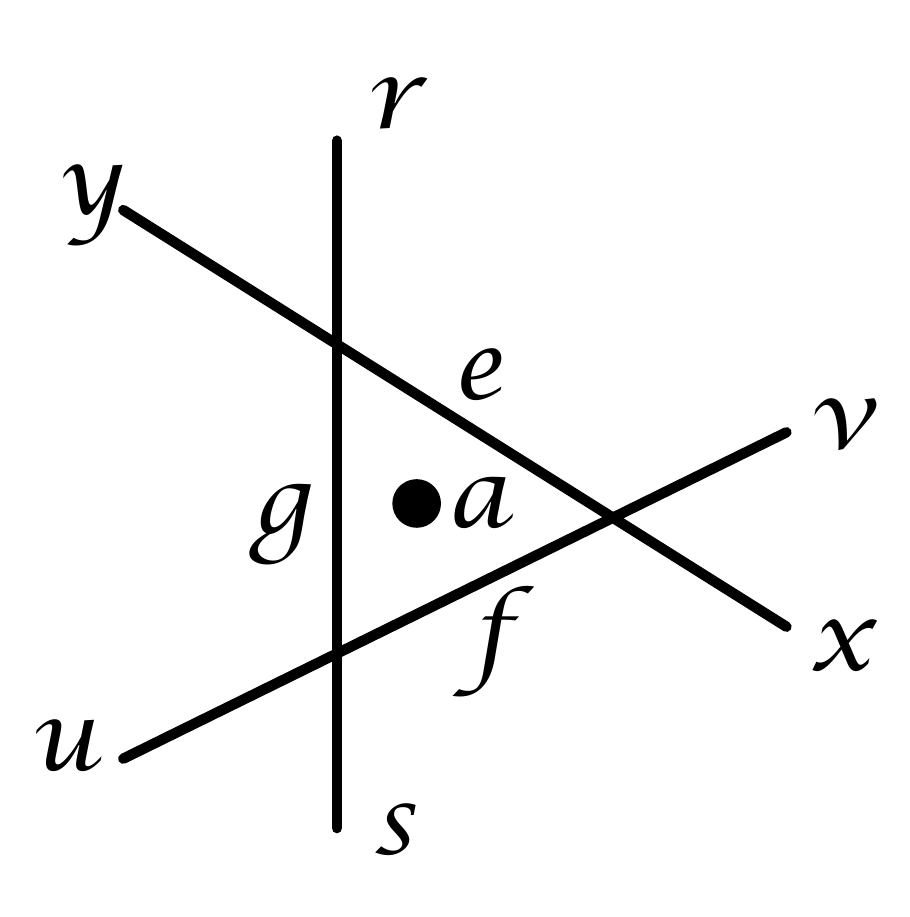}\hskip .5truein
\includegraphics[scale=.4]{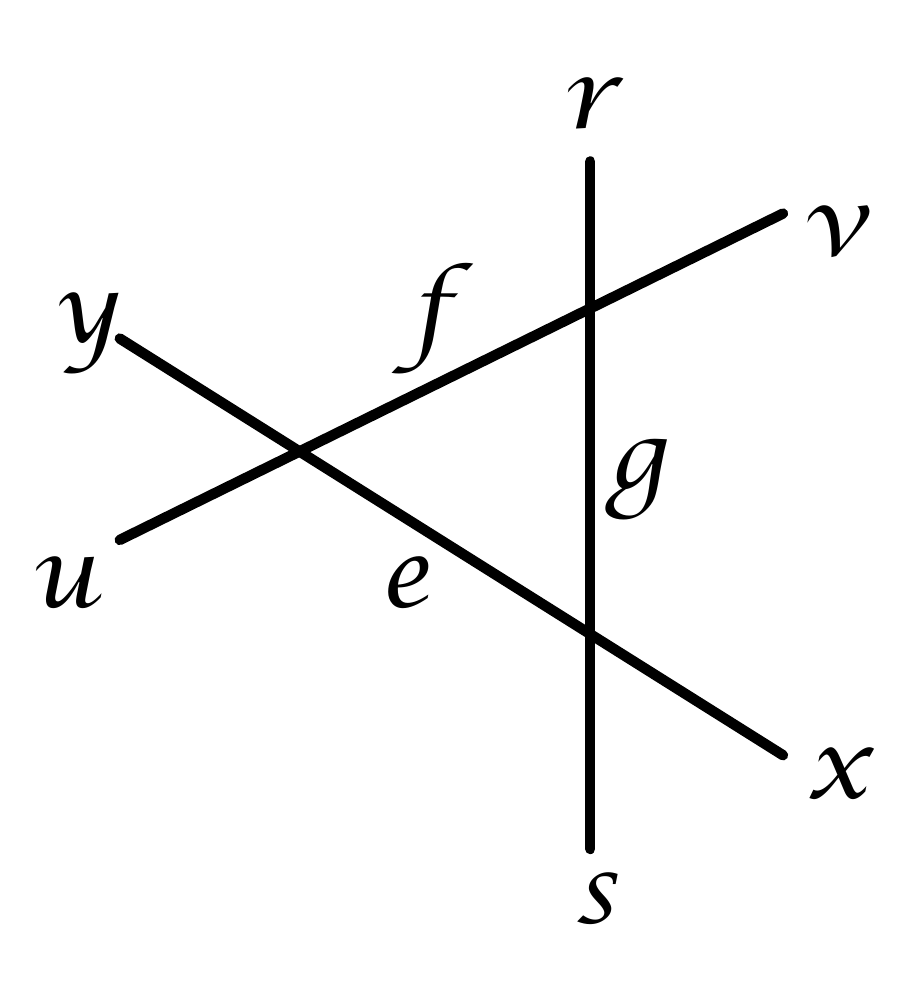}
\caption{The Reidemeister triangle in $D_1$ and $D_2$.}\label{fg:vxInReidTriangle1}
\end{center}
\end{figure}

In the $K_4$ induced by $\{u,v,x,y\}$, $a$ is in the 3-face $T^1_{u,y}$ bounded by $uy\times_{e,f}^1$ and, therefore, in the discs bounded by the 3-cycles $uyx$ and $yuv$ that do not contain $D_1[v]$ and $D_1[x]$, respectively.  By \ref{it:rotationK4}, this holds true also for $D_2$.  Analogous statements hold for the other two $K_4$\!'s involving two of the three edges from $e,f,g$.  

Using the labelling described above for $D_2$, the faces $T^{\changes{\textrm\,2}}_{u,y}$, $T^{\changes{\textrm\,2}}_{r,v}$, and $T^{\changes{\textrm\,2}}_{x,s}$ are bounded by $uy\times_{e,f}^2$, $rv\times_{f,g}^2$, and $xs\times_{e,g}^2$, respectively. 
Moreover, $a$ is in all three of the faces $T^{\changes{\textrm\,2}}_{u,y}$, $T^{\changes{\textrm\,2}}_{r,v}$, and $T^{\changes{\textrm\,2}}_{x,s}$, so no two of them are disjoint.

In the same $K_4$ induced by $\{u,v,x,y\}$, in both $D_1$ and $D_2$, $yx$ crosses $uv$.  In $D_2$, as we traverse $uv$ from $u$,  we first travel along the boundary of $T^2_{u,y}$, then pass through $\times^2_{e,f}$, followed by $\times^2_{f,g}$, showing that $\times^2_{f,g}$ is separated by $uy\times^2_{e,f}$ from $a$ and, therefore, $T^2_{rv}$ is not contained in $T^2_{uy}$.

By symmetry, this works for all pairs from $T^2_{rv}$, $T^2_{uy}$, and $T^2_{xs}$.  Since no two are disjoint,  we deduce that any two of $rv\times_{f,g}^2$, $uy\times_{e,f}^2$, and $xs\times_{e,g}$ intersect.  Since they intersect each other an even number of times, they intersect each other at least twice.

Therefore, the 6-cycle $rvuyxs$  has at least nine crossings in $D_2$, consisting of the three that define $R$ and the at least six mentioned at the end of the preceding paragraph.  Since nine is the most crossings a 6-cycle can have in a good drawing, we conclude that it is exactly nine.  Thus, any two of \changer{$uy\times_{e,f}^2$, $rv\times_{f,g}^2$, and $xs\times_{e,g}^2$}  cross exactly twice.  Moreover, every pair of non-adjacent edges in the 6-cycle must cross.   In particular, $rv$ crosses $uy$.

When we consider the two crossings of $uy\times_{e,f}^2$ and $rv\times_{f,g}^2$, for example, one of them is $rv$ crossing $uy$.  Since $e$, $f$, and $g$ pairwise cross at the corners of $R$, no two of them can provide the second crossing of $uy\times_{e,f}^2$ and $rv\times_{f,g}^2$.  Therefore, the second crossing involves either $rv$ or $uy$.  That is, either $rv$ crosses \changer{$uy\times_{e,f}^2$} twice or $uy$ crosses \changer{$rv\times_{f,g}^2$} twice.  

Since these conclusions are symmetric, we may assume the former.  The final piece of information that we require is the order in which these two crossings occur.  By way of contradiction, suppose that, as we traverse $D_2[rv]$ from $D_2[v]$,  we first cross the \changer{$xy$-segment of $uy\times_{e,f}^2$} before crossing $uy$.  See Figure \ref{fg:wrongOrder}.

\begin{figure}[!ht]
\begin{center}
\includegraphics[scale=.4]{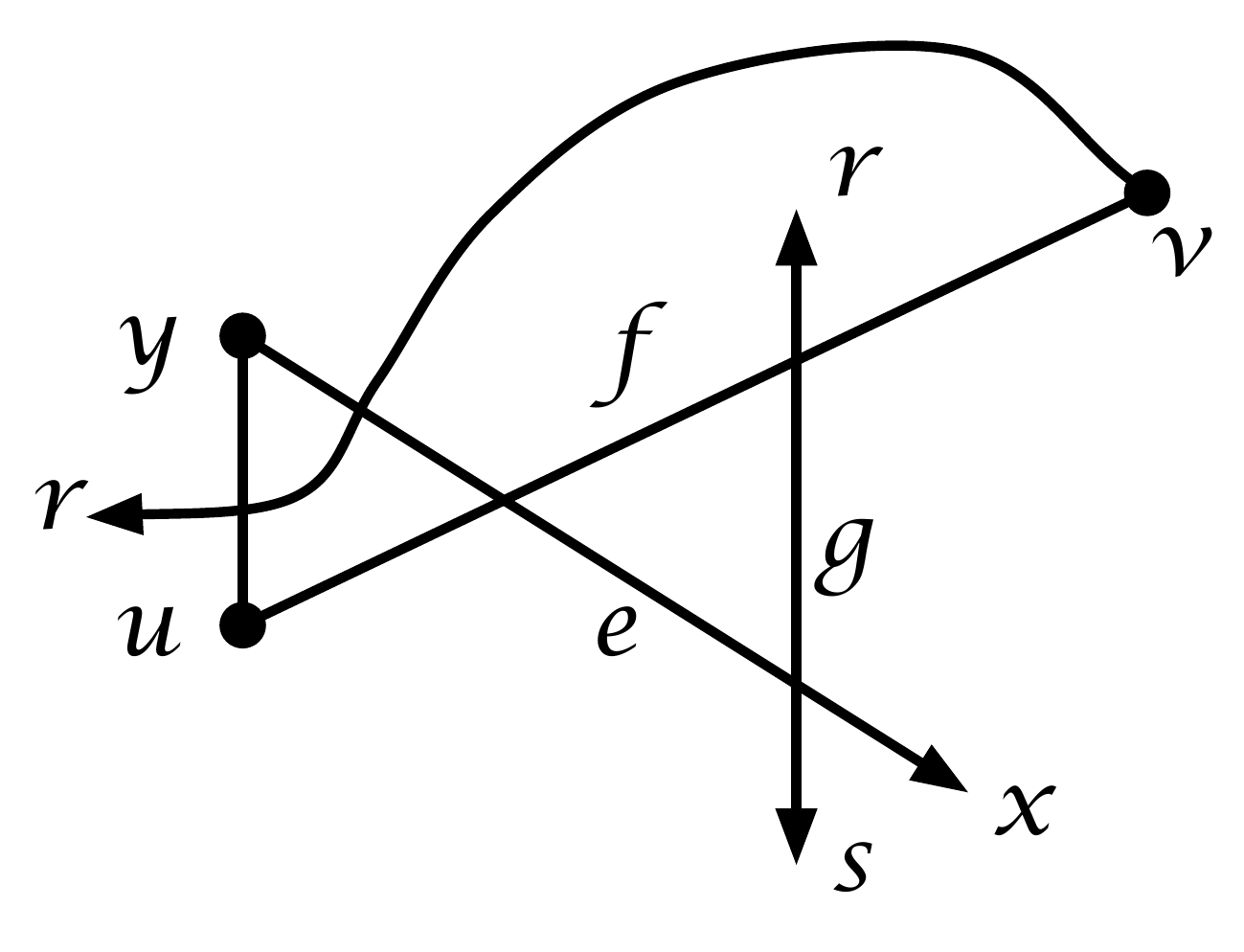}
\caption{$D_2[rv]$ crosses $T^2_{u,y}$ in the wrong order.}\label{fg:wrongOrder}
\end{center}
\end{figure}

Consider the simple closed curve $\Omega$ consisting of the arc in $D_2[uv]$ from $\times_{e,f}^2$ to $D_2[v]$, then along $D_2[rv]$ from $D_2[v]$ to the crossing of $D_2[rv]$ with the \changer{$xy$-segment of $uy\times_{e,f}^2$}, and then along $D_2[xy]$ back to $\times_{e,f}^2$.  

By goodness, the portion of $D_2[rs]$ from $\times_{f,g}^2$ to $D_2[r]$ cannot cross $\Omega$, so $D_2[r]$ is on the side of $\Omega$ that is different from the side containing the crossing of $rv$ with $uy$.  Again, goodness forbids the crossing of $\Omega$ with the portion of $D_2[rv]$ from $r$ to the crossing with $uy$.  This contradiction shows that the first crossing of \changer{$uy\times_{e,f}^2$} by $D_2[rv]$, as we start at $v$, is with $uy$.  See Figure \ref{fg:rightOrder}.

\begin{figure}[!ht]
\begin{center}
\includegraphics[scale=.4]{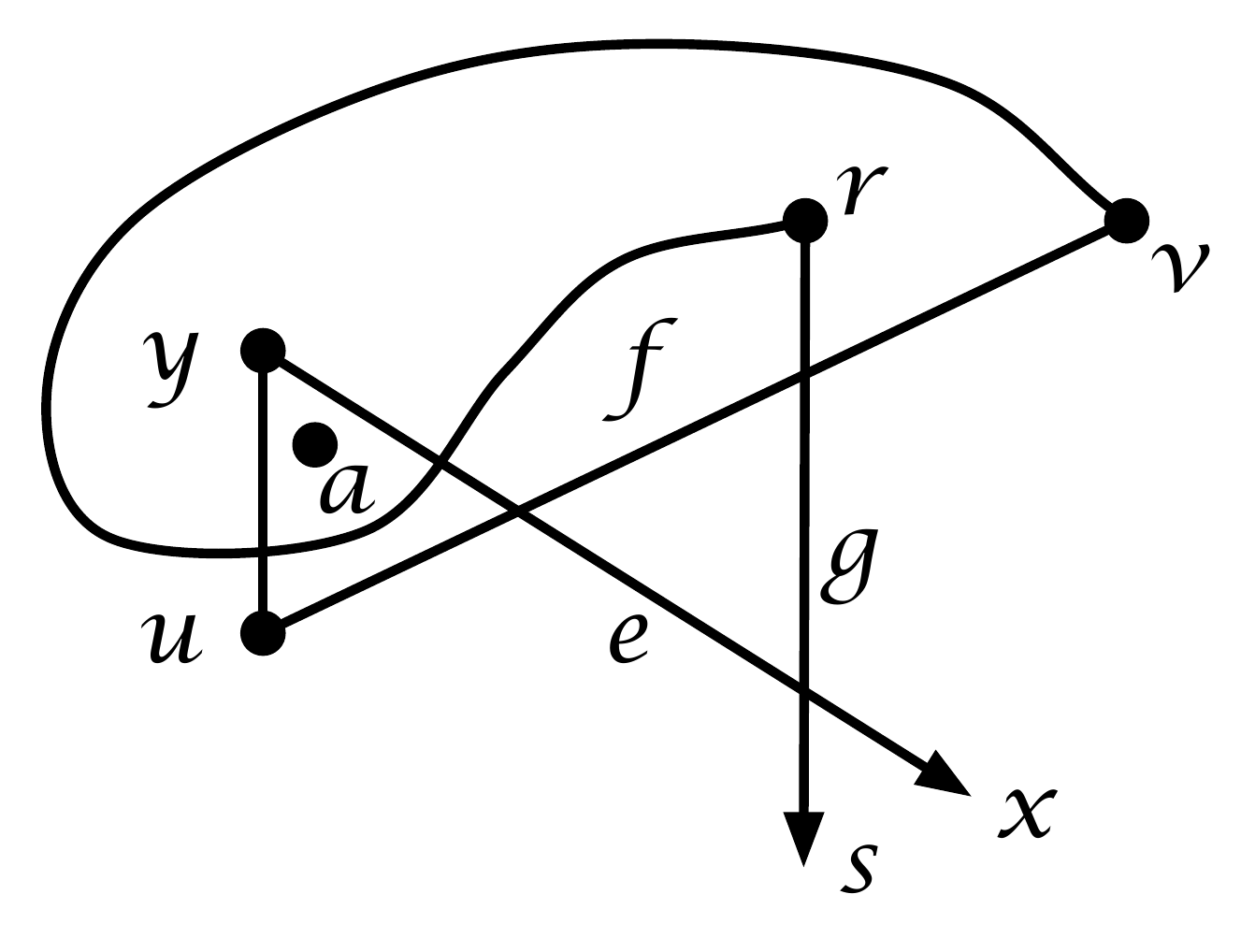}
\caption{This is how $D_2[rv]$ crosses $T^2_{u,y}$.}\label{fg:rightOrder}
\end{center}
\end{figure}

The vertex $a$ is in $T^2_{r,v}\cap T^2_{u,y}$.  As $D_1[a]$ and $D_1[y]$ are on different sides of  $D_1[R]$,  $D_1[ay]$ crosses at least one of $D_1[rs]$, $D_1[uv]$, and $D_1[xy]$.  Thus, \ref{it:rotationCrossing} implies  $D_2[ay]\not\subseteq T_{u,y}^2$.

Goodness implies that $D_2[ay]$ must cross the \changer{$uv$-segment of $uy\times_{e,f}^2$}.  In order to do that, it must cross $rv$ first.  But now $y$ and the crossing $\times$ of $D_2[ay]$ with $D_2[uv]$ are separated by the simple closed curve $\Omega'$ consisting of the portion of $uv$ from $\times_{e,f}^2$ to $v$, $rv$ from $v$ to its crossing with $xy$, and the portion of $xy$ between this crossing and $\times_{e,f}^2$.  

However, the portion of $ay$ from $\times$ to $y$ cannot cross any of the three parts of $\Omega'$, because \changer{each} part is contained either in an edge incident with $y$ or is crossed by the complementary part of $ay$.  This contradiction completes the proof.
\end{cproof}

We are now ready to prove Gioan's Theorem.   The structure of our proof is very much the same as that given by the algorithm in \cite{gioan}.

\begin{cproofof}{Theorem \ref{th:gioan}}
Label the vertices of $K_n$ as $v_1,v_2,\dots,v_n$.  For each $i=1,2,\dots,n$, let $K_i$ denote the complete subgraph induced by $v_1,v_2,\dots,v_i$.   We shall show, by induction on $i$, that there is a sequence $\Sigma_i$ of Reidemeister moves so that, if $D^i_1$ is the drawing of $K_n$ obtained by performing the moves $\Sigma_i$ on $D_1[K_n]$, then there is an orientation-preserving homeomorphism of the sphere that maps $D^i_1[K_i]$ onto $D_2[K_i]$ (of course preserving the labels $v_1,\dots,v_i$).  

The claim is trivial for $i<4$ and is  \ref{it:rotationK4} for $i=4$.  Thus, we may assume $i\ge 5$ and the result holds for $i-1$.  In particular, replacing $D_1$ with $D^{i-1}_1$, we may assume $D_1[K_{i-1}]$ is the same as $D_2[K_{i-1}]$.  For ease of notation and reference, we will use $K_{i-1}$ to also denote this common drawing of $K_{i-1}$.  We may assume that, for $\alpha=1,2$, $D_\alpha[K_i]$ is obtained from $K_{i-1}$ by using dual paths for each edge $v_iv_j$ ($j\in \{1,2,\dots,i-1\}$), together with a segment in the last face to get from the dual vertex in that face to $v_j$. 

This understanding needs a slight refinement, since, for example, it is possible for two edges incident with $v_i$ to use the same sequence of faces (in whole or in part).  Thus, as dual paths, they would actually use the same segments.  We allow this, as long as the two edges do not cross on the common segments.  They can be slightly separated at the end to reconstruct the actual drawing.
  
Since each face of $K_{i-1}$ is the intersection of all the discs bounded by 3-cycles that contain the face, \ref{it:triangleContainment} shows that $v_i$ is in the same face of $K_{i-1}$ in both $D_1$ and $D_2$.
If there is an orientation-preserving homeomorphism of the sphere that maps $D_1[K_i]$ onto $D_2[K_i]$, then we are already done, so we may assume there is some least $j\in \{1,2,\dots,i-1\}$ such that $D_1[v_iv_j]$ and $D_2[v_iv_j]$ use different dual paths in $K_{i-1}$.  Let $F_1,F_2,\dots,F_r$ be the faces of $K_{i-1}$ traversed by $D_2[v_iv_j]$.

Each $F_k$ is (essentially) a union of faces of $D_1[K_n]$.  The (planar) dual of the graph in $F_k$ is connected, so there are  paths  in each $F_k$ to obtain a dual path in $D_1[K_n]$ that restricts to the dual path of $K_{i-1}$ representing $D_2[v_iv_j]$.  We will refer to this dual path in $D_1[K_n]$ as $D^*_2[v_iv_j]$.  Our objective will be to find a sequence of Reidemeister moves in $D_1[K_n]$ to make a drawing $D^j_1[K_n]$ such that there is an orientation-preserving homeomorphism of the sphere to itself that maps $D^j_1[K_{i-1}]$ plus the edges $v_iv_1,\dots,v_iv_j$ onto $D_2[K_{i-1}]$ plus the edges $v_iv_1,\dots,v_iv_j$.

The construction shows that $D_1[v_iv_j]\cup D^*_2[v_iv_j]$ is a closed curve $C^i_j$ with finitely many common segments (which might just be $v_i$, $v_j$ and single dual vertices).  In particular, $C^i_j$ divides the sphere into finitely many regions.

\begin{claim}\label{cl:noVertexSeparation} All the vertices of $K_{i-1}-\{v_j\}$ are in the same region of $C^i_j$.
\end{claim}

\begin{proof}
Let $x$ and $y$ be vertices of $K_{i-1}-\{v_j\}$.  If $xy$ does not cross $D_1[v_iv_j]$, then it also does not cross $D_2[v_iv_j]$; thus it also does not cross $D^*_2[v_iv_j]$.  It follows that $xy$ is disjoint from $C^i_j$, showing that $x$ and $y$ are in the same region of $C^i_j$.

Thus, we may assume that $xy$ crosses $D_1[v_iv_j]$.  Then it also crosses $D_2[v_iv_j]$ and, therefore, $D^*_2[v_iv_j]$.  Letting $J$ be the $K_4$ induced by $v_i,v_j,x,y$, both $D_1[J]$ and $D_2[J]$ have $v_iv_j$ crossing $xy$.  There is a unique face $F$ of $D_1[J]$ bounded by a 4-cycle in $J$.  There is an $xy$-arc $\gamma$ in $F$ that goes very near alongside the path $P=(x,v_j,y)$ and is disjoint from $D_1[v_iv_j]$. 

As the rotations are the same, $D_1[v_iv_j]$ and $D^*_2[v_iv_j]$ both  start in the same angle of $v_j$ in $K_{i-1}$.  Thus, $\gamma$ is also disjoint from $D^*_2[v_iv_j]$, so $x$ and $y$ are in the same region of $C^i_j$. \end{proof}

A {\em $j$-digon\/} is a simple closed curve in $C^i_j$ consisting of a subarc of $D_1[v_iv_j]$ and a subarc of $D^*_2[v_iv_j]$.  If $D_1[v_iv_j]\ne D^*_2[v_iv_j]$, then some point $z$ of $D^*_2[v_iv_j]$ is not in $D_1[v_iv_j]$.  Traverse in both directions in $D^*_2[v_iv_j]$ from $z$ until first reaching $D_1[v_iv_j]$; adding the segment of $D_1[v_iv_j]$ between these two points produces a $j$-digon.   By Claim \ref{cl:noVertexSeparation}, each $j$-digon $C$ bounds a closed disc that is disjoint from $\{v_1,v_2,\dots,v_{i-1}\}$; this is the {\em clean side\/} of $C$.

To complete the induction, we show that there is a sequence $\Gamma_{i,j}$ of Reidemeister moves such that, in the drawing $D^{i,j}_1[K_n]$  obtained by doing the sequence $\Gamma_{i,j}$ to $D_1[K_n]$, $D^{i,j}_1[K_{i-1}]= D_2[K_{i-1}]$ and also all the edges $v_iv_1,\dots, v_iv_j$ are the same in both $D^{i,j}_1[K_i]$ and $D_2[K_i]$.
Since $D_1[v_iv_j]$ and $D^*_2[v_iv_j]$ use different dual sequences (relative to $K_{i-1}$), there is a $j$-digon.   

Lemma \ref{lm:crossingOrderDetermines} shows that the edges $D_1[v_iv_j]$ and $D^*_2[v_iv_j]$ cross the same edges of $K_{i-1}$, but not in the same order.    Among all the $j$-digons, let $C$ be one having a minimal clean side $S$.  Thus, no other $j$-digon has its clean side contained in $S$.  If $xy$ is an edge of $K_{i-1}-\{v_j\}$ that intersects $S$, then Claim \ref{cl:noVertexSeparation} implies $xy\cap S$ consists of a single arc having one end in $D_1[v_iv_j]$ and the other end in $D^*_2[v_iv_j]$.  

Lemma \ref{lm:reidemeister} shows that there is a sequence $\Pi$ of Reidemeister moves in $K_{i-1}\cup C$, each involving $D_1[v_iv_j]$, that removes all crossings from $S$; at that point $C\cap D_1[v_iv_j]$ and $C\cap D^*_2[v_iv_j]$ use the same dual path (relative to $K_{i-1}$).    We prove that there is a sequence $\Pi'$ of Reidemeister moves that apply to $D_1[K_n]$ and performs the same effect, but in $D_1[K_n]$, of making $C\cap D_1[v_iv_j]$ use the same dual path (relative to $K_{i-1}$) as $C\cap D^*_2[v_iv_j]$.  The sequence $\Pi'$ includes $\Pi$ as a subsequence; the remaining moves in $\Pi'$ all involve some edge not in $K_{i-1}$ and not among the edges $v_iv_1,\dots,v_iv_j$.  In particular, these additional moves do not affect the drawing of either $K_{i-1}$ or the edges $v_iv_1,\dots,v_iv_j$.  This is clearly enough to complete the induction.

Suppose $\Pi=\pi_1,\pi_2,\dots,\pi_r$ and, for some $s\in\{1,\dots,r\}$, we have found such a sequence $\Pi'_{s-1}$ of moves that contains $\pi_1,\pi_2,\dots,\pi_{s-1}$ as a subsequence; we may suppose $\Pi'_{s-1}$ terminates with $\pi_{s-1}$.  In particular, $\Pi'_0$ is the empty sequence.  Let $D^{s-1}_1[K_n]$ be the drawing of $K_n$ obtained by performing the sequence $\Pi'_{s-1}$ on $D_1[K_n]$.  

The move $\pi_s$ consists of operating on a Reidemeister triangle $R_s$ inside $S$ involving the three edges $e,f,g$.  For each move in $\Pi$, and in particular for $\pi_s$, one of the edges is $D_1[v_iv_j]$; we choose $e$ to be this edge.  Thus, $f$ and $g$ are in $K_{i-1}$.  The move $\pi_s$ involves moving the crossing of $f$ with $g$ across $e$ so that it is now outside $S$.  Therefore, $f$ and $g$ cross inside $R_s$ and so $f$ and $g$ cross $C\cap D^{s-1}_1[v_iv_j]$ and $C\cap D^*_2[v_iv_j]$ in different orders.  Thus, $R_s$ is a Reidemeister triangle for the drawings $D^{s-1}_1[K_n]$ and $D_2[K_n]$.

Lemma \ref{lm:differentOrderReidTriang} shows that no vertex of $K_n$ is inside $R_s$.  None of the edges in $K_{i-1}$ and $\{v_iv_1,\dots,v_iv_j\}$ goes into $R_s$.  Every other edge intersects each side of $R_s$ at most once and intersects $R_s$ an even number of times.  Every other edge that crosses $R_s$ makes a pre-Reidemeister triangle inside $R_s$.     We claim that there is a sequence $\Omega$ of Reidemeister moves that empties $R_s$ and involves moving only these other edges.  

An easy induction shows that if $\alpha$ and $\beta$ cross inside $R_s$, then there is a sequence of Reidemeister moves available to push their crossing over any of the edges $e,f,g$ that they both cross.  

Thus, there is a sequence $\Omega$ of Reidemeister moves that involves moving only these other edges and that empties $R_s$, at which point we may perform the move $\pi_s$.  Thus, $\Pi'_s=\Pi'_{s-1}\Omega\pi_s$ is the required sequence of moves on $D_1[K_n]$.  

It follows that there is a sequence $\Theta$ of Reidemeister moves on $D_1[K_n]$ that produces a drawing $D'_1$ of $K_n$ such that $D'_1[v_iv_j]$ and $D^*_2[v_iv_j]$ have the same dual sequence with respect to $K_{i-1}$.  Therefore,  $D'_1[v_iv_j]$ and $D_2[v_iv_j]$ have the same dual sequence with respect to $K_{i-1}$.  Lemma \ref{lm:crossingOrderDetermines} implies that there is an orientation-preserving homeomorphism of the sphere to itself that maps $D'_1[K_{i-1}+\{v_iv_1,\dots,v_iv_j\}]$ to $D_2[K_{i-1}+\{v_iv_1,\dots,v_iv_j\}]$, as required.  

Thus, by induction on $j$ there is a sequence of Reidemeister moves on $D_1[K_n]$ to make a new drawing $D'_1[K_n]$ such that there is an orientation-preserving homeomorphism of the sphere to itself that maps $D'_1[K_{i}]$ to $D_2[K_{i}]$.  Finally, induction on $i$ shows that there is a sequence of Reidemeister moves on $D_1[K_n]$ to produce a drawing $D^*_1[K_n]$ and an orientation-preserving homeomorphism of the sphere to itself that maps $D^*_1[K_{n}]$ to $D_2[K_{n}]$, which is precisely Theorem \ref{th:gioan}.
\end{cproofof}

\bigskip

{\bf \Large Acknowledgments}

\bigskip
We appreciate the friendly discussions with Emeric Gioan.  We thank Stefan Felsner for his insightful comments; in particular, he pointed out a significant oversight in our original proof.

\end{document}